\documentclass[11pt,leqno]{article}

\usepackage[utf8x]{inputenc} 	
\usepackage[english]{babel} 	
\usepackage{ucs}             	
\usepackage{hyperref}			    

\usepackage{geometry} 
\usepackage{graphicx}     
\usepackage{xcolor}       
\usepackage{pstricks-add} 
\usepackage{enumitem}     
\usepackage{verbatim}     

\usepackage{amsmath, amsthm}
\usepackage{amsfonts, amssymb}
\usepackage{stmaryrd}     
\usepackage{euscript}     

\hypersetup{
	pdfstartview={FitH},	
    pdfauthor={Thomas Rey}, 
}

\theoremstyle{plain}
	\newtheorem{theorem}{Theorem}[section]
	\newtheorem*{theorem*}{Theorem}
	\newtheorem{lemma}{Lemma}[section]
	\newtheorem{proposition}{Proposition}[section]
	\newtheorem*{proposition*}{Proposition}

	\numberwithin{equation}{section}

\theoremstyle{remark}
	\newtheorem{remark}{\textbf{Remark}}[]

\theoremstyle{definition}
	\newtheorem{definition}{Definition}[]
	
\newcommand{\eg}{\emph{e.g. }}

\newcommand{\SSS}{\mathbb{S}}

\author{ \textsc{Thomas Rey}\thanks{Universit\'e de Lyon, Universit\'e Claude Bernard Lyon 1, Institut Camille Jordan, 43 bd. du 11 Novembre 1918, 69622 Villeurbanne cedex, France (rey@math.univ-lyon1.fr).}}

\title{Blow Up Analysis for Anomalous Granular Gases\thanks{The author would like to express his gratitude for the support of the European Research Council ERC Starting Grant 2009, project 239983-NuSiKiMo;}}
\date{Final version}

\begin{document}	
	\pagestyle{myheadings}
	\thispagestyle{plain}
	\markboth{\textsc{Thomas Rey}}{\textsc{Anomalous Granular Gases}}
	
	\maketitle

	\begin{abstract}
		We investigate in this article the long-time behaviour of the solutions to the energy-dependant, spatially-homogeneous, inelastic Boltzmann equation for hard spheres. This model describes a diluted gas composed of hard spheres under statistical description, that dissipates energy during collisions.
		We assume that the gas is  ``anomalous'', in the sense that energy dissipation increases when temperature decreases. This allows the gas to cool down in finite time. 
		We study existence and uniqueness of blow up profiles for this model, together with the trend to equilibrium and the cooling law associated, generalizing the classical Haff's Law for granular gases. 
		To this end, we investigate the asymptotic behaviour of the inelastic Boltzmann equation with and without drift term by introducing new strongly ``nonlinear'' self-similar variables.
	\end{abstract}

	\section{Introduction}
		
			We are interested in this paper in a particular model of granular gases. 
			A granular gas is a set of particles which interact by energy dissipating contact interactions. 
			This is a quite different model than perfect molecular gas where energy is conserved, because the particles are ``microscopic'' regarding the scale of the system, but ``macroscopic'' in the sense that there are not molecules at all, but rather grains of a given size. 
			For example, a suspension of pollen in a fluid or a planetary ring~\cite{Kawai:1990,Araki:1986} can be seen as a granular gas when intermolecular forces are neglected, each one in a different scale. 
			
			Our goal is to investigate the global behaviour of the kinetic energy of a space-homogeneous gas of inelastic hard spheres interacting \emph{via} binary collisions (that is a granular gas), and existence and uniqueness of blow up profiles associated. 
			The study of this kind of system started with the physics paper \cite{haff:1983} of P.K. Haff, and has since generated a large increase of interest, both in Physics (a complete introduction of the subject can be found in the textbook \cite{brilliantov:2004} by N. Brilliantov and T. P\"oschel) and Mathematics (\emph{cf.} the review article \cite{Villani:2006} by C. Villani). 
			It involves some complicated phenomena, such as kinetic \emph{collapse} (\emph{cf.} K. Shida \cite{Shida:1989}) where the gas freezes completely, \emph{clustering} at hydrodynamic level for inhomogeneous gases, or even spontaneous loss of homogeneity (as proved by I. Goldhirsch and G. Zanetti in~\cite{Goldhirsch:1993}). 
			This article will especially deal with the  case of a gas whose particle's collision rate increases with dissipation of energy. 
			Such a gas is sometimes called \emph{anomalous} (see article \cite{Mischler:20061} of S. Mischler and C. Mouhot), because of the unusual behaviour of this collision rate. 
			Another model of anomalous gas has been introduced by G. Toscani in~\cite{Toscani:2000} by assuming that the collisions are close to be elastic for large relative velocities of the grains.
			It has been studied precisely by H. Li and G. Toscani in~\cite{Li:2004}, showing the finite time extinction of the support.
			
			A granular gas can be described in a purely Newtonian way, but the number of macroscopic particles involved (ranging from $10^6$ to $10^{10}$) leads to adopt a statistical physics' point of view. Therefore, we shall study the so-called space-homogeneous \emph{inelastic Boltzmann} equation (also known as \emph{granular gases} equation), given by
			\begin{equation}
				\frac{\partial f}{\partial t} = Q_e(f, f) \label{homo},
			\end{equation}
			where $f = f(t,v)$ represents the particle's distribution function, depending on time $t \geq 0$ and velocity $v \in \mathbb{R}^d$. 
			It is always a nonnegative function.
			The collision operator $Q_e(f, g)$, which will be made precise in the following (particularly concerning the meaning of subscript $e$), models a binary inelastic collision process of hard spheres type, localized in time. This equation will be supplemented with the initial value
			\begin{equation} \label{CI}
				f(0, \cdot) = f_{in}.
			\end{equation}
			
			\begin{figure}
				\begin{center} 
					\psset{xunit=1.0cm, yunit=1.0cm, dotstyle=o, dotsize=3pt 0, linewidth=0.8pt, arrowsize=3pt 2, arrowinset=0.25}
					\begin{pspicture*}(-3.36,-4.54)(8.26,4.95)
						\pscircle[linewidth=1.4pt](0,0){2}
						\pscircle[linewidth=1.4pt](4,0){2}
						\psline[linewidth=1.6pt]{->}(3,0)(1,0)
						\psline[linewidth=1.4pt]{->}(8,3)(5.16,1.63)
						\psline[linewidth=1.4pt]{->}(-3,-4)(-1.37,-1.45)
						\psline[linewidth=1.4pt,linestyle=dashed,dash=4pt]{->}(5.16,-1.63)(8,-3)
						\psline[linewidth=1.4pt,linestyle=dashed,dash=4pt]{->}(-1.37,1.45)(-3,4)
						\psline[linewidth=1.4pt]{->}(-1.37,1.45)(-1.93,3.74)
						\psline[linewidth=1.4pt]{->}(5.16,-1.63)(7.01,-3.36)
						\rput[bl](2.43,0.24){$\omega$}
						\rput[bl](6.79,2.19){$v$}
						\rput[bl](-1.99,-2.88){$v_*$}
						\rput[bl](-1.48,2.59){$v'_*$}
						\rput[bl](5.68,-2.83){$v'$}
					\end{pspicture*} 
				\end{center}
				\label{figCol}
				\caption{Geometry of collision (dashed lines are elastic and solid lines inelastic)}
			\end{figure}
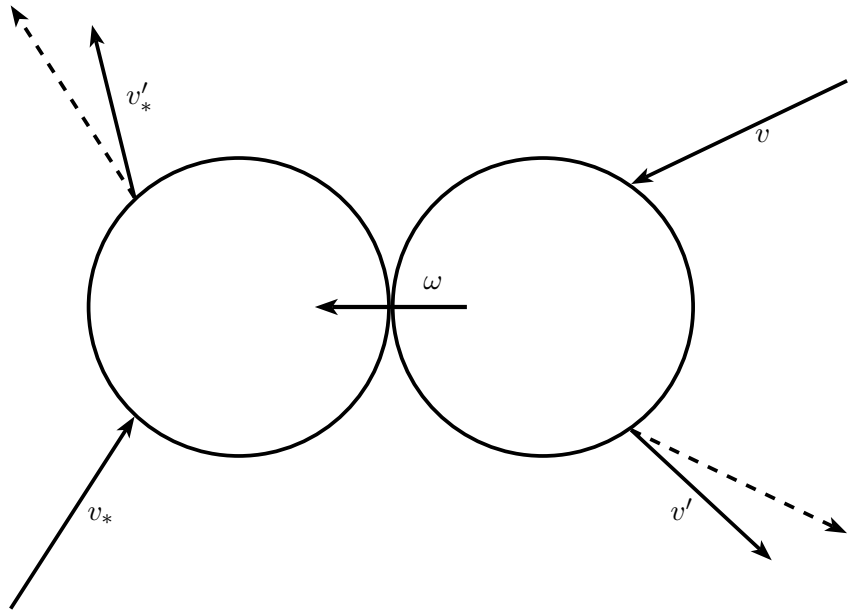
						
			The inelasticity is characterized by a collision mechanics where mass and momentum are conserved and kinetic energy is dissipated. At microscopic level, it can be described as follows: given two particles of pre-collisional velocities $v$ and $v_*$, their respective post-collisional velocities, denoted by $v'$ and $v_*'$, are given by 
			\begin{equation*}
				\left \{\begin{aligned} 
					v' & = v - \frac{1 + e}{2} \left( u \cdot \omega \right) \omega, \\
					v_*' & = v_* + \frac{1 + e}{2} \left( u \cdot \omega \right) \omega,
				\end{aligned}\right .
			\end{equation*}
			where $u := v - v_*$  is the \emph{relative velocity} of a pair of particles, $ \omega$ is the \emph{impact direction} and $e \in [0,1]$ the dissipation parameter, known as \emph{restitution coefficient} (see Figure \ref{figCol} for a sketch of the collision process).
			
			Physically, it means that energy is dissipated in the impact direction only. The parameter $e$ can depend on relative velocity and kinetic energy of the particles. For the need of the analysis, we will simply assume that $e$ is a constant, which is a rather good physical approximation (a complete discussion on this topic can be found in \cite{brilliantov:2004}). Besides,  $e$ will also be taken different from 1, since $e=1$ concerns the classical elastic case, in which no dissipation occurs. 
			The inelastic case $e < 1$ can also be characterized by the fact that the impact direction $\omega$ does not bisect the angle between pre and post-collisional velocities. 
			
			The parametrization of post-collisional velocities can also be found by using some properties of the model. Indeed, it is equivalent to conservation of impulsion and dissipation of energy:
			\begin{align*}
				v' + v_*' & = v + v_*, \\
				|v'|^2 + |v_*'|^2 - |v|^2 - |v_*|^2 & = - \frac{1-e^2}{2} | u \cdot \omega |^2 \leq 0.
			\end{align*}
			
			One of our goal will be the study of the macroscopic dissipation of energy, sometimes referred to as the cooling process of the gas. 
			This is of high interest when one wants to study the self-similar behaviour of dissipative kinetic equations.
			
			This phenomenon has been investigated mathematically by S. Mischler and C. Mouhot in \cite{Mischler:20062} for the case of constant restitution coefficient $e$ and by R. Alonso and B. Lods in \cite{alonso:2009} for visco-elastics hard spheres, more realistic at the physical level. 
			In the latter one, energy dissipation is a power law of the relative velocity of colliding particles. 
			We will use ideas of these two articles to write some of the results we present in the following.
			The paper \cite{Li:2004} also deals with the cooling process of a one dimensional nonlinear friction equation, which arises when dealing with quasi-elastic limit of equation \eqref{homo} with variable restitution coefficient.
			Finally, J.A. Carrillo, M. Di Francesco and G. Toscani in~\cite{Carrillo:2006Self} and J.A. Carrillo and J.L. V\'azquez in \cite{Carrillo:2007a} used the behaviour of the kinetic energy to investigate intermediate asymptotics  of nonlinear diffusion equation.
		
			Let us now describe precisely the collision operator we consider throughout the rest of this article.
			Let $f$ and $g$ be two nonnegative particles distribution functions only depending on $v \in \mathbb{R}^d$. 
			The collision operator $Q_e(f,g)$, where $e \in [0,1)$ is the constant restitution coefficient, can be expressed in the following weak form: given a regular test function $\psi$,
			\begin{equation} \label{Qweak}
				\langle Q_e(f,g), \psi \rangle := \frac{1}{2 }\int_{\mathbb{R}^d \times \mathbb{R}^d \times \mathbb{S}^{d-1}} f_{*} \, g \, \left(\psi' + \psi_*' - \psi - \psi_* \right) B\left( |u|, \widehat u \cdot \omega, \mathcal{E}(f) \right) d\omega \, dv \, dv_*,
			\end{equation}
			where we have used the usual shorthand notation $\psi' := \psi(v')$, $\psi_*' := \psi(v_*')$, $\widehat u := u/|u|$ and 
			\begin{equation*}
				\mathcal{E} \left( f \right) := \int_{\mathbb{R}^d}|v|^2 f(v) dv 
			\end{equation*}
			denotes the \emph{kinetic energy} of $f$. Moreover, $B$ is a positive function known as the \emph{collision kernel}. 
			We will  assume that it factorizes as:
			\begin{equation} \label{ba}
				B\left (|u|, \widehat u \cdot \omega, \mathcal{E}\right ) = |u| \, b\,( \widehat u \cdot \omega) \, \mathcal{E}^{-a},
			\end{equation}
			for a nonnegative constant $a$ and a nonnegative function $b$ (the collisional \emph{cross section}) of mass $1$ in the unit sphere, bounded by below and above by two positive constants $\beta_1$ and $\beta_2$:
			\begin{equation}\label{HS}
				\forall x \in [-1, 1], \ 0 < \beta_1 \leq b(x) \leq \beta_2 < \infty.
			\end{equation}
			
			The gas is anomalous thanks to this collision kernel.
			Indeed, according to \eqref{ba}, the collision frequency will increase when the kinetic energy decreases.
			As we will see in the following, this kinetic effect will leads to the cooling in finite time, provided that the coefficient $a$ is big enough.
			Such a phenomenon is related to the one described in \cite{Toscani:2000,Li:2004}: the gas is anomalous because the grains are close to be elastic for large relative velocities.
			This is in opposition with the more physically intuitive models where collisions are elastic for small relative velocities (for example viscoelastic hard spheres, as described in~\cite{brilliantov:2004, alonso:2009} or even the simplified threshold model of T. P\"oschel, N.V. Brilliantov and T. Schwager of~\cite{poschel:2005transient}).
			It also yields a finite time cooling.
			
			Concerning this topic, we can also mention some recent works of I. Fouxon \emph{et al.}~\cite{Fouxon:2007a} and  Kolvin \emph{et al.}~\cite{Kolvin:2010}. 
			Under a weak inelasticity assumption with constant restitution coefficient, these authors derived granular hydrodynamics (of Euler and Navier-Stokes type) and proved the finite time cooling in these models, for a one parameter family of initial conditions.
			To do so, they used a Lagrangian formulation to exhibit non self-similar solution to these equations, and we shall not consider this approach in our paper.

			We can also give a strong form of the collision operator.
			As it has been pointed out in \cite{Villani:2006} (the full derivation of this expression can be found \eg in articles \cite{bobylev:2000,bobylev:2001} of A. Bobylev, J. Carrillo and I. Gamba), one can write
			\begin{align*}
				Q_e(f, g)(v) & = \mathcal{E}(f)^{-a} \int_{\mathbb{R}^d \times \mathbb{S}^{d-1}} |u| \left( J \frac{|'u|}{|u|} \ 'f \ 'g_* - f g_* \right) b(\widehat u \cdot \sigma) \, d \sigma \, dv_* \\
				 & = Q_e^+(f,g)(v) - f(v) \,L(g)(v).
			\end{align*}
			In this expression, $\, 'v$ and $\, 'v_*$ are the pre-collisional velocities of two particles of given velocities $v$ and $v_*$, defined for $\sigma \in \SSS^{d-1}$ as
			\begin{equation*}
				\left\{\begin{aligned} 
					& 'v = \frac{v + v_*}{2} - \frac{1 - e}{4e} (v-v_*) + \frac{1+e}{4e} |v - v_*| \sigma \\
					& 'v_* = \frac{v + v_*}{2} + \frac{1 - e}{4e} (v-v_*) - \frac{1+e}{4e} |v - v_*| \sigma.
				\end{aligned} \right.
			\end{equation*}
			The parameter $\sigma$ is no longer the impact direction, but the center of what we can call the \emph{collision sphere}~\cite{Villani:2006}.
			The operator $Q_e^+(f,g)(v)$ is usually known as \emph{gain} term because it can be seen as the number of particles of velocity $v$ created by collisions of particles of pre-collisional velocities $\, 'v$ and $\, 'v_*$, whereas $f(v)\,L(g)(v)$ is the \emph{loss} term, modeling the loss of particles of pre-collisional velocities $v$. 
			Finally, the coefficient $J$ is the Jacobian of the map $(v, v_*) \to (\, 'v,\, 'v_*)$, not identically equal to 1, but since $e$ is a nonzero constant given by \[ J = \frac{|u|}{e^2 |'u|}.\]
						
			Taking successively $\psi(v) = 1, \, v$ and $|v|^2$ in \eqref{Qweak} gives macroscopic conservation of mass and momentum, and dissipation of kinetic energy. This implies that equilibria of this collision operator are Dirac distributions $\rho \, \delta_{v=v_0}$ of prescribed mass $\rho$ and momentum $\rho \, v_0$. In order to prevent apparition of these trivial solutions, one can add a diffusive term to give an input of energy, as was studied for example by Cercignani, R. Illner and C. Stoica in \cite{cercignani:2001} and by A. Bobylev and Cercignani in \cite{bobylev:2002diff} for Maxwell molecules.
			It is also possible to look for self-similar solutions to this equation thanks to a rescaling which leads to study the inelastic Boltzmann equation with drift term, as was done by  M. Ernst and R. Brito in \cite{ernst:2002}.
			
			Let us now describe more precisely the asymptotic behaviour of the energy of $f$, assuming that this distribution is time dependant, of mass $1$ and zero momentum. 
			
			\begin{definition}
				If $f$ is solution to \eqref{homo}, we call \emph{cooling process} the asymptotic behaviour of $\mathcal{E}(f)(t)$. We say that there is a \emph{blow up} when $\mathcal{E}(f)(t) \to 0$ for $t \to T_c$, where $T_c$ is the time of explosion of $f$.
			\end{definition}
			
			If one multiplies equation \eqref{homo} by $|v|^2$ and integrates in velocity, using the weak formulation \eqref{Qweak} of the collision operator and \eqref{ba}, one gets dissipation of kinetic energy:
			\begin{equation} \label{edoE1}
				\frac{d}{dt} \mathcal{E}(f)(t) = - \mathcal{E}(f)(t)^{-a} \iint_{\mathbb{R}^d \times \mathbb{R}^d}f(t,v) f(t,v_*) |u| D \left( |u| \right) dv \, dv_*,
			\end{equation}
			where $D$ is a nonnegative quantity usually called \emph{dissipation rate}, given by
			\begin{equation*}
				D(|u|) = \frac{1-e^2}{4} \int_{\mathbb{S}^{d-1}} |u \cdot \omega|^{2} \, b(\widehat u \cdot \omega) \, d\omega.
			\end{equation*}
			Using the  polar coordinates $\cos (\theta) =  \widehat u \cdot \omega$, we have
			\begin{equation} \label{eqDu}
				D(|u|) = \frac{1-e^2}{4} |u|^2 \left | \mathbb S^{d-2} \right | \int_{0}^{\pi} \cos^2(\theta) \sin^{d-3}(\theta) \, b(\cos(\theta)) \, d\theta.
			\end{equation}
			
			Provided that $f$ is of mass $1$ and zero momentum, one can then apply Jensen's inequality to $\Psi(v)= |v|^3$ and to the probability measure $f(v_*)\, dv_*$, together with conservative properties of equation \eqref{homo} to get
			\begin{equation*}
				\int_{\mathbb{R}^d}f(v_*) \Psi(|v-v_*|) dv_* \geq \Psi \left( \left| v \int_{\mathbb{R}^d}f(v_*) dv_* - \int_{\mathbb{R}^d} v_* f(v_*) dv_* \right| \right) = \Psi(|v|).
			\end{equation*}
			Using equalities \eqref{edoE1} and \eqref{eqDu}, it comes that
			\begin{equation} \label{ineqEtmp1}
				\frac{d}{dt} \mathcal{E}(f)(t) \leq - \beta_e \, \mathcal{E}(f)(t)^{-a} \int_{\mathbb{R}^d} f(t,v) |v|^{3} dv,
			\end{equation}
			where $\beta_e$ is given by 
			\begin{equation*}
				\beta_e = \frac{1-e^2}{4} \left | \mathbb S^{d-2} \right|  \int_{0}^{\pi} \cos^2(\theta) \sin^{d-3}(\theta) \,b(\cos(\theta)) \, d\theta < +\infty.
			\end{equation*}			
			Moreover, applying Hölder inequality with $s = 3/2$ and its conjugate $s' = 3$ to the map $v \mapsto f(v) |v|^2$, one gets thanks to mass conservation
			\begin{equation} \label{ineqEtmp2}
				\left( \int_{\mathbb{R}^d} f(v) |v|^2 dv \right)^{\frac{3}{2}} \leq \int_{\mathbb{R}^d} f(v) |v|^{3} dv.
			\end{equation}
			Using Fubini Theorem for nonnegative measurable function and gathering inequalities \eqref{ineqEtmp1} and \eqref{ineqEtmp2}, we finally find a closed differential inequality for $\mathcal{E}(f)$, namely
			\begin{equation}\label{ineqE}
				\frac{d}{dt} \mathcal{E}(f)(t) \leq - \beta_e \, \mathcal{E}(f)(t)^{-a + \frac{3}{2}}, \text{ for }t < T_c.
			\end{equation}
			
			According to this inequality, the cooling process will depend on the nonnegative  coefficient $a$ introduced in the definition of the collision kernel \eqref{ba}.
			We now have to describe rigorously the spaces of solutions to \eqref{homo}, in order to introduce the corresponding Cauchy problem and the main results of this article.
			
		\subsection{Functional Framework and Main Results}
		
			Let us present some functional spaces needed in the paper. We denote by $L_q^1$ for $q \in [1, +\infty)$ the following weighted Lebesgue space
			\begin{equation*}
				L^1_q = \left\{ f : \mathbb{R}^d \rightarrow \mathbb{R} \text{ measurable; }\|f\|_{L^1_q} := \int_{\mathbb{R}^d} |f(v)| \, \langle v\rangle^{q} \, dv < \infty \right\},
			\end{equation*}
			where $\langle v\rangle := \sqrt{1+ |v|^2}$. Thanks to this definition, we can introduce the sets of distribution of given mass $1$, zero momentum and energy $\mathcal E > 0$ as
			\begin{equation*}
				\mathcal G := \left\{ f \in L^1_1 : \int_{\mathbb{R}^d} f(v) \, dv = 1, \int_{\mathbb{R}^d} f(v) \, v \, dv = 0 \right \}
			\end{equation*}
			and
			\begin{equation*}
				\mathcal G_\mathcal{E} := \left\{ f \in \mathcal G : \int_{\mathbb{R}^d} f(v) \, |v|^2 \, dv = \mathcal{E} \right \}.
			\end{equation*}
			Then, we denote by $W_q^{1,1}$ the weighted Sobolev space
			\begin{equation*}
				W_q^{1,1} := \left\{ f \in L^1_q;  \nabla f \in L^1_q \right \}.
			\end{equation*}
			We also introduce the space $BV_q$ of weighted bounded variations functions, given as	the set of weak limits in $\mathcal D'$ of sequences of smooth functions which are bounded in $W_q^{1,1}$. Finally, $\mathcal{M}^1(\mathbb{R}^d)$ is the space of probabilities measures on $\mathbb{R}^d$.
			
			The Cauchy problem (\ref{homo}-\ref{CI}) has been extensively investigated in the past few years, starting with paper \cite{bobylev:2000} and then by A. Bobylev, Cercignani and G. Toscani in \cite{bobylev:2003}, where Maxwellian molecules were considered using Fourier transform. 
			A lot of properties of a more general model which includes elastic and inelastic collisions for Maxwellian molecules as a particular case (but not hard spheres) have also been studied in the chapter of book \cite{bobylev:2008} by A. Bobylev, Cercignani and I. Gamba.
			
			Concerning hard spheres kernel, paper~\cite{Toscani:2000} presents a one dimensional model where the restitution coefficient depends on the relative velocity of the colliding particles through a phenomenological inverse power law.	
			Then, the first important result about tail behaviour in any dimension was given by A. Bobylev, I. Gamba, and V. Panferov in~\cite{bobylev:2004}.
			Existence of smooth solution to a model with stochastic heating was proven by I. Gamba, V. Panferov, and C. Villani
 in~\cite{gamba:2004}.
			S. Mischler and C. Mouhot finally proved existence and uniqueness of weak solutions to the Cauchy problem for \eqref{homo} in \cite{Mischler:20061}.
			The more physical space inhomogeneous setting was investigated by D. Benedetto and M. Pulvirenti in~\cite{Benedetto:2002} for the model introduced in~\cite{Toscani:2000}, but in one dimension of space and velocity.
			Finally, R. Alonso studied in~\cite{Alonso:2009vacuum} the Cauchy problem near vacuum for a large class of restitution coefficients, in any dimension of space and velocity.
			
			Let us define more precisely what we call Cauchy problem for \eqref{homo}.

			\begin{definition}
				Given a nonnegative initial condition $f_{in} \in L^1_2 \cap \mathcal G\,$, a nonnegative function $f$ defined on $[0,T] \times \mathbb{R}^d$ is said to be solution to the Cauchy problem \eqref{homo}-\eqref{CI} if
				\begin{equation*}
					f \in \mathcal C \left( 0,T; L^1_2 \right),
				\end{equation*}
				and if \eqref{homo}-\eqref{CI} holds in the weak sense, namely,
				\begin{equation*}
					\int_0^T \left( \int_{\mathbb{R}^d} f(t,v) \, \partial_t \psi (t,v) \, dv + \left \langle Q_e(f,f)(t,\cdot), \psi \right\rangle \right) dt = \int_{\mathbb{R}^d} f_{in}(v) \, \psi(0,\, \cdot) \, dv,
				\end{equation*}
				for any $\psi \in \mathcal C_c^1 \left( [0,T] \times \mathbb{R}^d \right)$.
			\end{definition}
			
			We notice that it is always possible to assume the initial condition in $\mathcal G$, namely of mass $1$ and zero momentum, since we may always reduce to this case  by a scaling and translation argument, using invariance properties of the collision operator.
			
			Let us now introduce the \emph{cooling time} $T_c$ of $f$ as the time before blow up, namely
			\begin{equation*}
				T_c := \sup \, \{T > 0 : \mathcal{E}(f)(t) > 0, \, \forall \, t < T \}.
			\end{equation*}
			For a collision kernel given by \eqref{ba} and \eqref{HS}, Theorem 1.4 of \cite{Mischler:20061} (recalled in the Appendix as Theorem \ref{thmCauchyPb}) states well-posedness in $L^1_3$ of the Cauchy problem for the granular gases equation, together with the existence of the cooling process when $t \to T_c$.

			We now state main results of this paper concerning cooling process and blow up solutions to the anomalous model. It is known from \cite{Mischler:20061} that if $a$ is greater than $1/2$, $T_c$ is finite, meaning that blow up of $f$ occurs in finite time whereas it takes infinite time for $a \leq 1/2$. 
			We give the accurate asymptotic behaviour of the energy in the main theorems of this article, for all nonnegative values of the parameter a, in agreement with the formal computation \eqref{ineqE}.

			\begin{theorem} \label{thmCrit}
				Let the collision kernel be subject to assertions \ref{ba} and \ref{HS}. 
				If $f$ is solution to \eqref{homo} with an initial condition $0\leq f_{in}\in L_3^1 \cap \mathcal{G}_{1} \cap L^p$ for $p > 1$, there exist some positive constants $C_i$, $i\in \{1,...,6\}$ and $T_c$, depending on $e$, $a$ and $f_{in}$ such that
				\begin{enumerate}[label=(\roman{*})]
					\item \emph{sub-critical case:} if $0 \leq a < 1/2$, and then $\alpha = 1/(2 \, a - 1) < 0$,  \label{thmCoolingInf}
						\begin{equation*}
							\frac{1}{(C_1 \, t + 1)^{-2 \, \alpha}} \leq \mathcal{E}(f)(t) \leq \frac{1}{(C_2 \, t + 1)^{-2 \, \alpha}},\, \forall \, t >0;
						\end{equation*}
					\item \emph{critical case:} if $a = 1/2$,  \label{thmCoolingExp}
						\begin{equation*}
							e^{-C_3 \, t} \leq \mathcal{E}(f)(t) \leq  e^{-C_4 \, t} ,\, \forall \, t >0;
						\end{equation*}
					\item \emph{super-critical case:} if $a > 1/2$, and then $\alpha = 1/(2 \, a - 1) > 0$, \label{thmCoolingFin}
						\begin{equation*}
							\left( -C_5 \, t + 1\right)^{2 \, \alpha} \leq \mathcal{E}(f)(t) \leq \left(-C_6 \, t + 1 \right)^{2 \, \alpha}, \, \forall \, t < T_c.
						\end{equation*}
				\end{enumerate}
				Moreover, for all $a \geq 0$, there exist a function $V \in \mathcal C^1(0,T_c)$ and a nonnegative profile $G \in L_3^1 \cap \mathcal{G}_{1}$ such that the distribution
				\begin{equation*}
					F(t,v) := V(t)^{d} \, G(V(t) \, v), \, \forall \, (t,v) \in (0,T_c) \times \mathbb{R}^d,
				\end{equation*}
				is solution to \eqref{homo}, called \emph{self-similar solution}.	
			\end{theorem}

			We propose in Section \ref{secUnifProofCrit} a proof of this Theorem based on the analysis of self-similar solutions to the inelastic Boltzmann equation with a linear drift term (widely studied in \cite{Mischler:20062,mischler:20091}) and the asymptotic analysis of the time derivative of the energy. To this end, we introduce a new self-similar scaling, nonexplicit and strongly nonlinear (energy-dependant, as in~\cite{Carrillo:2006Self}).
					
			\begin{remark}
				One can check that the point \ref{thmCoolingInf} of this theorem is in good agreement with the one proven in \cite{Mischler:20062} with $a = 0$, and which is classically known as \emph{Haff's Law}\cite{haff:1983}: 
				\begin{equation*}
					\frac{m_1}{(1+t)^{2}} \leq \mathcal{E}(f)( t ) \leq \frac{M_1}{(1+t)^{2}}.
				\end{equation*}
			\end{remark}

			We then state a theorem of existence and uniqueness (up to a translation of time) of self-similar solutions to equation \eqref{homo} with $a \geq 0$. We also obtain convergence of solutions to \eqref{homo} bounded in $L^1_3$ toward these particular solutions.
			
			\begin{theorem} \label{thmAttract}
				Let the collision kernel be subjects to assertions \eqref{ba} and \eqref{HS} with $a \geq 0$. There exists a constructive $e_* \in (0,1)$ such that for all $e \in [e_*, 1)$,
				\begin{enumerate}[label=(\roman{*})]
					\item the self-similar profile $G = G_e \in L_3^1 \cap \mathcal{G}_{1}$ is unique and if $ F_e$ and $\bar F_e$ are two self-similar solutions to \eqref{homo}, there is a time $t_0 < T_c$ such that $\bar F_e (t,v) = F_e (t+t_0,v)$ for $t > \max \{0,-t_0\}$.
					\item For $p > 1$ and any $M_0>0$, there exists $\eta \in (0,1)$ such that if
						\begin{equation*}
							\left\{ \begin{aligned}
								& f_{in} \in L^1_3 \cap \mathcal{G}_{1} \cap L^p, \\
								& \| f_{in} \|_{L^1_3}\leq M_0, \end{aligned} \right.
						\end{equation*}
						the solution $f$ to equation \eqref{homo} satisfies for a nonnegative constant $C$
						\begin{equation*}
							\| f(t, \cdot) - F_e(t, \cdot) \|_{L^1} \leq \left \{
							\begin{aligned}
								& \frac{C}{( 1+C_2 \, t)^{- \alpha  \, \mu_e}} & \text{ if } a < 1/2, \\
								& C e^{- C_4 \, \mu_e \, t/2} & \text{ if } a = 1/2, \\
								& C ( 1-C_6 \, t )^{\alpha \, \mu_e} & \text{ if } a > 1/2,
							\end{aligned} \right.
						\end{equation*} 
						where $\alpha = 1/(2 \, a - 1)$ and $\mu_e = (1-\eta) + \mathcal O(1-e)$.
				\end{enumerate}			
			\end{theorem}
			
			The proof of this Theorem is also based on the analysis of self-similar solutions to the inelastic Boltzmann equation with drift term, together with the use of the explicit cooling processes of Theorem \ref{thmCrit}.
			
		\subsection{Plan of the Paper}
		
			The article is organized as follows. 			
			We prove Theorem \ref{thmCrit} in Section \ref{secUnifProofCrit}. 
			We give a nonlinear rescaling that allows us to treat the three different cooling process at once in Subsection \ref{subNLScal}. 
			We then use the expression of this rescaling to state a relation between energy in classical and self-similar variables. 
			Subsequently, we show in Subsection \ref{subLemma} results concerning the asymptotic behaviour of solutions to the granular gases equation with drift term, namely that their energy and time derivative are uniformly bounded in time.
			We finally use this result in Subsection \ref{subUnifProofCool} to establish the rigorous cooling process. 
			
			We then apply in Section \ref{secSelfSimUniqu} this precise asymptotic behaviour together with a Theorem of uniqueness of self-similar profiles already proven in \cite{mischler:20091} of S. Mischler and C. Mouhot to show the Theorem \ref{thmAttract} about rate of convergence toward self-similar solutions. 
			
			Finally, we recall in the Appendix some important results we used in the paper.
				
	\section{Proof of Theorem \ref{thmCrit}} 
		\label{secUnifProofCrit}
		
		This section presents the proof of Theorem \ref{thmCrit} which studies the cooling process for sub-critical ($a < 1/2$), critical ($a = 1/2$) and super-critical ($a > 1/2$) cases. To this end, we will prove some new estimates for the granular gases equation, with and without drift term. Let us start by introducing some nonclassical self-similar variables.
		
		\subsection{Nonlinear Self-similar Variables} \label{subNLScal}
		
			We shall use a rescaling (seen \eg in \cite{ernst:2002} and \cite{Mischler:20062}), in order to prevent the blow up of $f$ by ``zooming'' on the distribution and studying it in self-similar variables. 
			
			We will assume that the collision kernel is given by \eqref{ba} and  \eqref{HS} with $a \geq 0$. 
			The granular gases equation \eqref{homo} then reads
			\begin{equation} \label{homo2}
				\frac{\partial f}{\partial t}(t,v) = \mathcal{E}(f)(t)^{-a} Q_e(f, f)(t,v),
			\end{equation}
			where the collision operator is given for  test functions $\psi$ by
			\begin{equation*}
				\langle Q_e(f,g), \psi \rangle = \frac{1}{2}\int_{\mathbb{R}^d \times \mathbb{R}^d \times \mathbb{S}^{d-1}} |u| f_{*} \, g \, \left(\psi' + \psi_*' - \psi - \psi_* \right)  b (\widehat u \cdot \omega) d\omega \, dv \, dv_*. 
			\end{equation*}			
			The rescaling of the distribution $f$ will be written according to \cite{ernst:2002} as
			\begin{equation} \label{cgVar}
				\left\{ \begin{aligned} 
					& f(t,v) = V(t)^d \,g(T(t), V(t) \, v), \\
					& V(0) = 1, \ T(0) = 0, \\
					& \lim_{t \to T_c} T(t) = \lim_{t \to T_c} V(t) = + \infty.
				\end{aligned}\right.
			\end{equation} 
			The functions $T$ and $V$ have to be\emph{ strictly increasing }in time for this rescaling to be well defined.
			
			We shall determine $T$ and $V$ (the \emph{self-similar variables}) and derive the equation followed by the distribution $g = g(s,w)$ with $s = T(t)$ and $w = V(t) \, v$. The term  $V(t)^d$ in front of $g$ is simply given by  mass conservation. Moreover, by making a change of variables in the expression of the collision operator and thanks to hard sphere collision kernel, one gets that
			\begin{equation*}
				Q_{e} \left( g(\lambda \, \cdot), g(\lambda \, \cdot) \right) (v) = \lambda^{-d-1} Q_{e} (g,g)(\lambda v).
			\end{equation*}
			Especially, the bilinearity of operator $Q_e$ together with \eqref{cgVar} yields
			\begin{equation} \label{QSelfSim}
				Q_e(f,f)(t, \cdot) = (V(t))^{d-1} Q_{e}(g,g)(s, \cdot).
			\end{equation}
	
			\begin{remark}
				Because of \eqref{QSelfSim}, our approach of the problem cannot be immediately extended to the case of variable restitution coefficient. 
				Indeed, if $e = e(|u \cdot \widehat{n}|)$ where $\widehat n$ is the unit vector joining the center of the particles (this case cover the classical \emph{visco-elastic hard spheres} model, see N. Brilliantov and T. Pöschel~\cite{brilliantov:2000a}), introducing the self-similar variables \eqref{cgVar} in relation \eqref{QSelfSim} gives
				\begin{equation*}
					Q_e(f,f)(t, \cdot) = (V(t))^2 Q_{\widetilde{e}(t, \cdot)}(g,g)(s, \cdot),
				\end{equation*}
				that is collision operator becomes time-dependant through a new restitution coefficient $\widetilde{e}$. 
				Finally, the new collision equation is not ``autonomous'', which prevents us to use some previous results on this equation as we will do in the following (\eg Haff's Law).
			\end{remark}				
				
			
			Now, if $f$ follows rescaling \eqref{cgVar}, its time derivative becomes 
			\begin{equation} \label{derFSelfSim}
				\partial_t f = (V(t))^{d-1} \left( T'(t) \, V(t) \, \partial_s g + V'(t) \, \nabla_w \cdot (w g) \right).
			\end{equation}
			Thanks to relations \eqref{QSelfSim} and \eqref{derFSelfSim}, if $f$ is solution to equation \eqref{homo2}, the distribution $g$ is solution to 
			\begin{equation}\label{fullRescaled}
				T'(t) \, V(t) \, \partial_s g + V'(t) \nabla_w \cdot (w g) = \mathcal{E}(f)(t)^{-a} \, Q_{e}(g,g).
			\end{equation}
			We shall get rid of the term  in \eqref{fullRescaled} involving a negative power of the kinetic energy to obtain a simpler equation, the classical homogeneous Boltzmann equation with an anti-drift term. 
			
			To this end, let us assume the rescaling to be ``nonlinear'' by asking it to depend on the energy of the solution itself:
			\begin{equation*}
				\left \{ \begin{aligned}
					& V'(t) = \tau\, \mathcal{E} \left( f \right)(t)^{-a}, \\
					& T'(t) \, V(t) = V'(t)/\tau,
				\end{aligned} \right.
			\end{equation*}
			where $\tau$ is a nonnegative parameter. 
			This idea is very close to the one used in \cite{Carrillo:2006Self} for nonlinear diffusion equations (sometimes referred to as the \emph{Toscani Map}).
			It was also extended in \cite{Carrillo:2007a} to show the apparition of chaotic behaviour for a precise (constructive) nonlinearity : the self-similar profile of this equation ``oscillates'' between Gaussian (heat equation) and Zel'dovich-Kompaneets-Barenblatt (porous medium equation) profiles. 
			
			The functions $V$ and $T$ are not explicit but one can see that they are well defined and fulfill the requirement of the scaling \eqref{cgVar}. 
			Indeed, the map $t \mapsto \mathcal{E} (f)(t)^{-a}$ is strictly increasing on $[0, T_c)$ thanks to the cooling process \eqref{asymptConv}, is $1$ when $t=0$ given that the energy of the initial distribution $f_{in}$ has been normalized to this value and tends to infinity when $t \to T_c$. 
			Moreover, one has $T(t) = \log(V(t))/\tau$ with $\tau \geq 0$ and then T is increasing on $[0, T_c)$ with $T(0) = 0$ and $\lim_{t\to T_c} T(t) = + \infty$.
			 
			 With such an expression for $V$ and $T$ plugged in \eqref{fullRescaled}, the distribution $g$ is solution to the following equation: 
			\begin{equation} \label{drift}
				\frac{\partial g}{\partial s} + \tau \, \nabla_w \cdot (w g) = Q_{e}(g,g).
			\end{equation}
			Actually, this is the granular gases equation for inelastic hard spheres with constant restitution coefficient, complemented with an \emph{anti-drift} term. This last term will act like an input of energy and will prevent the blow up of $g(s, \cdot)$ toward a Dirac mass when $s \to \infty$. 
			This equation has been thoroughly studied in articles \cite{Mischler:20062,mischler:20091}. 
			Especially, we know that the kinetic energy of $g$ is bounded by above and below by two positive constants and that there exists steady states $G$ for this equation.
			We summed up this result in the Appendix as Theorem \ref{thmMiMoE}.

			Summarizing, the distribution $f$ defined by
			\begin{equation*}
				\left\{ \begin{aligned} 
					& f(t,v) = V(t)^d \, g(T(t), V(t) \,v), \\
					& V'(t) = \tau \, \mathcal{E} (f)(t)^{-a}, \\
					& T(t) = \log \,(V(t))/\tau,
				\end{aligned}\right.
			\end{equation*} 
			is solution to the granular gases equation \eqref{homo2} for all times as soon as the function $g$ is solution to  drift-collision equation \eqref{drift}. Conversely, if $f$ is solution to \eqref{homo2}, one can associate a solution $g$ to rescaled equation \eqref{drift} by setting
			\begin{equation*}
				g\,(s,w) = e^{-d  \tau s} f \left( V^{-1}\left(  e^{\tau s}\right), e^{-\tau s} \, w \right).
			\end{equation*}
			Note that the inverse $V^{-1}$ of $V$ is well defined, by the discussion above.
			
			Moreover, using the change of variables $w = V(t) v$, the following relation between energies of $f$ and $g$ holds: 
			\begin{equation*}
				\mathcal{E} (g)(T(t)) = (V(t))^{2} \, \mathcal{E} (f)(t) .
			\end{equation*}
			Then, the function $V$ can be written for all $0 \leq t < T_c$ as
			\begin{equation} \label{eqVEfg}
				V(t) = \left( \frac{\mathcal{E} (g)(T(t))}{\mathcal{E}(f)(t)}\right)^{1/2}.
			\end{equation}
			Using this expression, the bounds \eqref{haffg} of the energy of $g$ and the raw cooling process \eqref{asymptConv}, one has another proof that $V(t) \to \infty$ and $T(t) = \log(V(t)) \to \infty$ when $t \to T_c$.
			
			Finally, if $G$ is a \emph{self-similar profile}, that is a stationary solution to \eqref{drift}, we may associate a self-similar solution $F$ to the original equation \eqref{homo2} by setting
			\begin{equation*}
				F(t,v) = V(t)^{d}G(V(t) v), \ \forall (t,v) \in (0,T_c) \times \mathbb{R}^d. 
			\end{equation*}
			Such a $G$ exists thanks to Theorem \ref{thmMiMoE}, which proves the last assertion of Theorem \ref{thmCrit}.
			
		\subsection{Preliminary Results} \label{subLemma}
		
			We will show in this Subsection two results concerning solutions to the granular gases equation with and without drift term, that we will need in order to prove rigorously the cooling process.
			
			Let $h = h(t,v)$ for $t \geq 0$ and $v \in \mathbb{R}^d$ be solution to the inelastic Boltzmann equation 
			\begin{equation}\label{boltzEq}
			 	\frac{\partial h}{\partial t} = Q_e(h,h),
			 \end{equation}
			 that is the collision equation \eqref{homo2} with $a = 0$. According to \cite{alonso:2009} or \cite{Mischler:20062}, the (sub-critical) Haff's Law holds:
			\begin{equation} \label{haffh}
				\frac{m}{(1+\mu_0 \, t)^2} \leq \mathcal{E}(h)(t)\leq \frac{M}{(1+\mu_0 \, t)^2}, \, \forall t > 0.
			\end{equation}
			An upper control of the third order moment of $h$ can be deduced from this law. Let $m_l(h)$ denotes the $2l$-th order moment of a distribution $h$, that is
			\begin{equation*}
				m_l(h) = \int_{\mathbb{R}^d} h(v) \, |v|^{2l} \, dv.
			\end{equation*}
			
			\begin{lemma} \label{lemMom2}
				Let $h$ be solution to equation \eqref{homo2} for $a=0$, with an initial condition $h_{in} \in \mathcal G \cap L^p$ for $p > 1$. There exists a nonnegative constant $\kappa$ such that if  $m_{3/2}(h_{in}) \leq \kappa$, then for all $t>0$,
				\begin{equation} \label{inegMom2}
					0 \leq m_{3/2}(h)(t) \leq \frac{\kappa}{\left( 1+\mu_0 \, t \right)^{3}}.
				\end{equation}
			\end{lemma}

			\begin{proof}
				Setting $m_l = m_l(h)$ and multiplying \eqref{homo2} with $a=0$ by $|v|^3$, one gets after integration in velocity space
				\begin{equation*}
					\frac{d}{dt} m_{3/2}(t) =  \int_{\mathbb{R}^d} Q_{e}(h,h)(t,v) \, |v|^{3} \, dv.
				\end{equation*}
				Some more informations are then needed regarding the third order moment of the collision operator. A.V. Bobylev, I.M. Gamba and V.A. Panferov had shown in \cite{bobylev:2004} (and this result was then extended to the case of bounded cross section in~\cite{alonso:2009}, see Appendix, Theorem \ref{thmBGPmoments}) the following estimate:
				\begin{equation*}
					\int_{\mathbb{R}^d} Q_{e}(h,h)(t,v) \, |v|^{3} \, dv \leq -(1-\gamma) \, m_2(t) + \gamma \, S_{3/2}(t),
				\end{equation*}
				where $ 0 < \gamma < 1$ and
				\begin{align*}
					S_{3/2} & = \sum_{k=1}^{\left[\frac{5}{4}\right]} \begin{pmatrix} 3/2 \\ k \end{pmatrix}(m_{k+1/2}\,  m_{3/2-k} + m_k \, m_{3/2-k+1/2}) \\
					 & = \frac{3}{2}(m_{3/2} \, m_{1/2} + m_1^{2}).
				\end{align*}
				By definition, we have $m_1 = \mathcal{E}(h)$. 
				Besides, by convexity, $m_2 \geq m_{3/2}^{4/3}$ and $m_1 \geq m_{1/2}^2$. 
				Thus, $m_{3/2}(t)$ verifies thanks to the sub-critical Haff's Law \eqref{haffh}:
				\begin{align*}
					\frac{d}{dt} m_{3/2}(t) & \leq - (1-\gamma) \, m_{3/2}(t)^{4/3} + \frac{3}{2} \gamma \left (\mathcal{E}(h)(t)^2 + m_{3/2}(t) \, \mathcal{E}(h)(t)^{1/2} \right ) \\
					& \leq - (1-\gamma) \, m_{3/2}(t)^{4/3} + \frac{3}{2} \gamma \left ( \frac{ M^2 }{\left(1+\mu_0 \, t\right)^4} + m_{3/2}(t) \, \frac{ M^{1/2} }{1+\mu_0 \, t}\right ).
				\end{align*}
				Then, one has
				\begin{align*}
					\frac{d}{dt}\left ( m_{3/2}(t) -\frac{\kappa}{ \left(1+\mu_0 t\right)^{3}}\right ) \leq & - (1 - \gamma) \, m_{3/2}(t)^{4/3} + \frac{3 \, \mu_0 \, \kappa}{\left(1+ \,\mu_0 \, t\right)^{4}} \\
					 & + \frac{3}{2} \gamma \left ( \frac{M^2}{\left( 1+\mu_0 \, t \right)^4} + m_{3/2} (t) \frac{ M^{1/2} }{1+\mu_0 \, t}\right ),
				\end{align*}
				where $\kappa$ is a nonnegative constant such that $m_{3/2}(0) \leq \kappa$.
				Then inequality \eqref{inegMom2} is fulfilled for $t=0$ and by continuity, the lemma is proved for $t < t_*$ with a positive time $t_*$. For $t = t_*$, one gets $m_{3/2}(t_*) = \kappa ( 1+\mu_0 t_* )^{-3}$, and the last differential inequality reads
				\begin{align*}
					\frac{d}{dt} \left ( m_{3/2}(t_*) -\frac{\kappa}{ \left(1+\mu_0  \,t_* \right)^{3}}\right ) \leq 
					 & \left( - (1 - \gamma) \, \kappa^{4/3} + C_1 \kappa + \frac{3}{2} \, \gamma \, M^2 \right) \frac{1}{\left(1+ \mu_0 \, t_*\right)^4},
				\end{align*}
				where $C_1 = 3 \left (\mu_0 + \gamma \sqrt M/2 \right )$. If the constant $\kappa$ is chosen large enough, the right hand side of this inequality is negative, which concludes the proof.
			\end{proof}
			
			Thanks to this estimate, we can now show that the time derivative of $\mathcal{E} (g)$ is uniformly bounded.
			
			\begin{proposition} \label{propLimDer}
				If $g$ is solution to the  inelastic Boltzmann equation with drift term \eqref{drift}, then there exists two constructive constants $- \infty < M_0 < 0 < M_1 < +\infty$ such that
				\begin{equation*}
					M_0 \leq \frac{d}{ds} \mathcal{E} (g)(s) \leq M_1, \ \forall \, s \geq 0.
				\end{equation*}
				
			\end{proposition}
			
			\begin{proof}			
				Given that $g$ is solution to~\eqref{drift}, it is possible to define a new distribution $h$ by setting
				\begin{equation*}
					h(t,v) := V_0(t)^d g(T_0(t), V_0(t) \,v),
				\end{equation*}
				where we defined as in \cite{Mischler:20061}
				\begin{equation*}
					\left \{\begin{aligned}
						& V_0(t) = 1+\mu_0 \, t, \\
						& T_0(t) = \log \,(1+\mu_0 \, t).
					\end{aligned}\right.
				\end{equation*}
				Then $h$ is solution to the inelastic Boltzmann equation \eqref{boltzEq}.
				Thanks to the expression of this rescaling, one hase
				\begin{equation*}
					\mathcal{E}(g)\left (T_0(t)\right ) = V_0(t)^2 \mathcal{E}(h)(t).
				\end{equation*}
				Differentiating this relation with respect to time gives
				\begin{align*}
					\frac{d}{dt}\mathcal{E}(g)\left (T_0(t)\right ) = A(t) + B(t),
				\end{align*}
				with
				\begin{equation*}
					A(t) = 2 \, \mu_0 (1+ \mu_0 \, t) \, \mathcal{E} (h)(t), \  B(t) = (1+ \mu_0 \, t)^2 \frac{d}{dt} \mathcal{E}(h)(t).
				\end{equation*}
				Therefore, using the sub-critical Haff's Law \eqref{haffh}, if $t > 0$, 
				\begin{equation} \label{ineqAhyp}
					 \frac{2 \, \mu_0 \, m}{1+\mu_0 \, t} \leq A(t) \leq \frac{2 \, \mu_0 \, M }{1+\mu_0 \, t}.
				\end{equation}
				Moreover, by the weak expression~\eqref{Qweak} of $Q(h,h)$, equation \eqref{boltzEq}  and~\eqref{edoE1} with $a=0$, one has on the one hand
				\begin{align}
					B(t) & = (1+\mu_0 \, t)^2 \int_{ \mathbb{R}^d } Q(h,h) \, |v|^2 \, dv  \notag\\
					 & = - \beta_e\, (1+\mu_0 \, t)^2 \iint_{ \mathbb{R}^d \times \mathbb{R}^{d}} h h_* \, |v - v_*|^3 \, dv \, dv_* \leq 0. \label{ineqBneg}
				\end{align}
				On the other hand, Fubini Theorem and mass conservation yield
				\begin{equation*}
					\iint_{ \mathbb{R}^d \times \mathbb{R}^{d}} h h_* \, |v - v_*|^3 \, dv \, dv_* \leq 8 m_{3/2}(h).
				\end{equation*}
				Then, using inequality \eqref{inegMom2} of Lemma \ref{lemMom2} in~\eqref{ineqBneg} allows to show that $B$ satisfies
				\begin{equation} \label{ineqBhyp}
					-  \frac{ 8 \, \kappa \, \beta_e}{ 1+\mu_0 \, t } \leq B(t)  \leq 0.
				\end{equation}
				Therefore, we can gather inequalities~(\ref{ineqAhyp}, \ref{ineqBneg}, \ref{ineqBhyp}) to write
				\begin{equation} \label{ineqdtEgT}
					\frac{2(\mu_0 \, m - 4 \, \kappa \, \beta_e)}{1+\mu_0 \, t} \leq \frac{d}{dt}\mathcal{E}(g)\left (T_0(t)\right ) \leq \frac{2 \,\mu_0 \, M}{1+\mu_0 \, t}.
				\end{equation}
				Using the chain rule, we have
				\begin{align*}
					\frac{d}{dt}\mathcal{E} (g) \, \left (T_0(t)\right ) & = T_0'(t) \frac{d}{ds}\mathcal{E} (g)(s) = \frac{\mu_0 }{1+\mu_0 \, t} \frac{d}{ds}\mathcal{E} (g)(s)
				\end{align*}	
				and we can finally write according to~\eqref{ineqdtEgT},
				\begin{equation*}
					2 \, m - \frac{8 \, \kappa \, \beta_e}{\mu_0} \leq \frac{d}{ds}\mathcal{E}(g)(s) \leq 2 M. 
				\end{equation*}	
				This concludes the proof.
			
			\end{proof}
			
			\begin{remark}
				We notice in the proof that the lower bound $M_0$ is not necessarily a nonnegative constant, allowing the energy of $g$ to have some inflection points.
			\end{remark}
			
			We are now ready to prove Theorem \ref{thmCrit}.
			
		\subsection{Anomalous Cooling Process} \label{subUnifProofCool}
		
			We will prove in this Subsection the cooling process of an anomalous gas for a nonnegative constant $a$ in the collision kernel \eqref{ba}, that is points \ref{thmCoolingInf},  \ref{thmCoolingExp} and \ref{thmCoolingFin} of Theorem \ref{thmCrit}. Both upper and lower bounds for the energy of $f$ will be obtained together, thanks to the results of Subsection \ref{subLemma}.
			
			For the sake of simplicity, let us denote for $t < T_c$
			\begin{equation*}
				E(t) := \mathcal{E} (f)(t) \text{ and } \bar E(t) := \mathcal{E} ( g)\left (T(t) \right),
			\end{equation*}
			and set $\tau = 1$. 
			We have already seen in~\eqref{eqVEfg} that $V(t)= \bar E(t)^{1/2} E(t)^{-1/2}$.
			Then differentiating this expression with respect to time yields
			\begin{align}
				E(t)^{-a} & = V'(t) \notag\\
				 & = \frac{\bar E'(t)}{2 (E(t) \, \bar E(t) )^{1/2}} - \frac{\bar E(t)^{1/2} \, E'(t)}{2 \, E(t)^{3/2}}. \label{anomEqtmp1}
			\end{align}
			Thanks to the expression of $\bar E$ and a chain rule, one has on the one hand
			\begin{align}
				\bar E'(t) & = \frac{d}{dt}\left ( \int_{\mathbb{R}^d}g(T(t),w) \,|w|^2 \, dw \right ) \notag \\
				 & = T'(t) \, \chi(t), \label{anomEqtmp2}
			\end{align}
			where we defined
			\begin{equation*}
				\chi := \frac{d}{ds} \left (\mathcal E(g) \right)\circ T.
			\end{equation*}
			On the other hand, the time derivative of $T$ is given by
			\begin{equation}
				T'(t) = \frac{V'(t)}{V(t)} = \frac{E(t)^{-a+1/2}}{\bar E(t)^{1/2}}.\label{anomEqtmp3}
			\end{equation}
			Finally, gathering relations \eqref{anomEqtmp1}, \eqref{anomEqtmp2}, \eqref{anomEqtmp3} and dividing by $E(t)^{-a}$, it comes that
			\begin{equation} \label{anomEq}
				\frac{\chi(t)}{\bar E(t)} -\bar E(t)^{1/2} \, E'(t) \, E(t)^{a-3/2} = 2.
			\end{equation}
			Thanks to the result of Proposition \ref{propLimDer}, there exists two finite constants $M_0 \, < 0 < M_1$ such that $M_0 \leq \chi(t) \leq M_1$.
			Thus, using the uniform bounds \eqref{haffg} of $\bar E$ and the nonpositivity of $M_0$, one has
			\begin{equation*}
				\frac{M_0}{c_0} \leq \frac{\chi(t)}{\bar E(t)} \leq \frac{M_1}{c_0}.
			\end{equation*}
			Therefore, gathering equality \eqref{anomEq}, bounds \eqref{haffg} and inequality~\eqref{ineqE} imply that
			\begin{equation} \label{inegDiffCool}
				-K := c_0^{-1/2}\left (\frac{M_0}{c_0} -2 \right )\leq E'(t) \, E(t)^{a-3/2} \leq - \beta_e <0,
			\end{equation}
			where $\beta_e$ is given by \eqref{eqDu} and $K > 0$ because we took $M_0 < 0 < 2 \, c_1$.
			Then, if $0 \leq a < 1/2$, the time integration of the two sides of inequality \eqref{inegDiffCool} between $0$ and $t$ and the fact that $E(0) = 1$  yield
			\begin{equation*}
				\frac{1}{(C_1 \, t + 1)^{-2 \, \alpha}} \leq E(t) \leq \frac{1}{(C_2 \, t + 1)^{-2 \, \alpha}},
			\end{equation*}
			with $\alpha = 1/(2 \, a - 1) < 0$, $C_1 = -2 \, \alpha \, K > 0$ and $C_2 = -2 \, \alpha \, \beta_e > 0$. 
			This is the point \ref{thmCoolingInf} of Theorem \ref{thmCrit}. 
			Moreover, if $a=1/2$, the same argument gives the point \ref{thmCoolingExp}, namely
			\begin{equation*}
				\exp(-K t) \leq E(t) \leq \exp(-\beta_e \, t).
			\end{equation*}
			Finally, taking $a >1/2$ gives the point \ref{thmCoolingFin} with $C_3 = 2 \, \alpha \, K > 0$ and $C_4 = 2 \, \alpha \, \beta_e > 0$:
			\begin{equation*}
				(1-C_3 \, t)^{2 \, \alpha} \leq E(t) \leq (1-C_4 \, t)^{2 \, \alpha}.
			\end{equation*}

			Let us now show the result concerning self-similar profiles of equation \eqref{homo}.
			
	\section{Proof of Theorem \ref{thmAttract}}
		\label{secSelfSimUniqu}
		
		For $a \geq 0$, we will study in this Section the uniqueness of self-similar profiles of the inelastic Boltzmann equation with drift \eqref{drift}. 
		The cooling process found in the previous Subsection will allow us to state a result concerning convergence of solutions to \eqref{homo} towards self-similar solutions.
		
		We shall use a Theorem concerning the convergence toward self-similar profiles for small inelasticity in the scaled granular gases equation \eqref{drift}, which has been shown in \cite{mischler:20091} (we recalled it for reader convenience in the Appendix, Theorem \ref{thmMiMoAttract}). 
		For this, let us set
		\begin{equation*}
			\tau = \tau_e := 1-e, 
		\end{equation*}
		in order to balance the dissipation of kinetic energy by the drift. 
		Thanks to this scaling, we have uniqueness and attractiveness of the self-similar profiles of equation \eqref{drift}.

		Let us show thanks to this Theorem the trend to self-similar solution of our problem.
		We have already seen in Subsection \ref{subNLScal} that if $g$ is solution to \eqref{drift} then $f$ is solution to \eqref{homo} with $f(t,v) = V(t)^d \, g(T(t), V(t) \, v)$, where
		\begin{equation} \label{selfSimVar}
			\left\{ \begin{aligned}
				V(t) & = \left( \frac{\mathcal{E} ( g)(T(t))}{\mathcal{E} ( f)(t)} \right)^{1/2},\\
				T(t) & = \frac{\log \,(V(t))}{\tau_e}.
			\end{aligned}\right.
		\end{equation}
		Thus, if $G_e$ is the unique self-similar profile of equation \eqref{drift}, one can find a self-similar solution to equation \eqref{homo} by setting $F_e(t,v) = V(t)^{d} \, G_e(V(t) \, v)$. The uniqueness of this solution up to a translation of time can be shown as in \cite{mischler:20091} to prove the first point of Theorem \ref{thmAttract}. 
		
		Moreover, the transformation $w \to V(t) \, v$ and the rate of convergence toward equilibrium \eqref{attractProf} give
		\begin{align}
			\| f(t, \cdot) - F_e(t, \cdot) \|_{L^1} & = V(t)^d \int_{\mathbb{R}^d} \left |  g(T(t),V(t) \, v) - G_e(V(t) \, v) \right | \, dv \notag\\
			 & \leq \left \| g(T(t),\cdot) - G_e \right \|_{L^1_2} \notag\\
			 & \leq e^{-(1-\eta) \, \nu_e \, T(t)}. \label{ineqAttract}
		\end{align}
		Besides, thanks to the expression \eqref{selfSimVar} of the self-similar variables and the choice $\tau_e = 1-e$, one has
		\begin{equation*}
			\nu_e \, T(t) = \log \,(V(t)) + \mathcal O(1-e). 
		\end{equation*}
		The positive lower bound of $\mathcal{E} \left( g_{T(t)} \right)$ in Theorem \ref{thmMiMoE} together with inequality \eqref{ineqAttract} and relation \eqref{selfSimVar} yield
		\begin{equation*}
			\| f(t, \cdot) - F_e(t, \cdot) \|_{L^1} \leq C \mathcal{E}(f)(t)^{\mu_e/2},
		\end{equation*}
		where $\mu_e = 1-\eta + \mathcal O(1-e)$ and $C$ is a nonnegative constant. Finally, using the cooling process of Theorem \ref{thmCrit} that we have shown in last Subsection, we can conclude the proof of Theorem \ref{thmAttract}, that is the trend to self-similar solution of the solutions to \eqref{homo} and the rate of convergence depending on the (weak) inelasticity $1-e$ and the coefficient $a$ of the energy dependent collision kernel \eqref{ba}.
		
	\section*{Summary and Perspectives}
	
		We have given in this paper the asymptotic behaviour of the space-homogeneous inelastic Boltzmann equation for anomalous gases. Depending on a parameter, we can observe in this model a blow up in finite time. We quantified the time of blow up and gave the associated self-similar profiles.
		Under a weak inelasticity hypothesis, we also obtained the uniqueness (up to a translation of the time) of the self-similar solutions to this equation, and the convergence of the classical solutions toward them.
		
		To prove these theorems, we introduced a new energy-dependant self-similar scaling, which leads to the study of the inelastic Boltzmann equation with a linear drift term. We gave some results concerning the asymptotic behaviour of the energy of the solutions to this equation, by using some well known theorems about this equation.
		
		Concerning the perspectives of this work, we would like to adapt our results to nonconstant restitution coefficient models, such as the viscoelastic one. We also want to prove the rate of cooling using a more classical self-similar scaling, generalising the one used in \cite{ernst:2002} and \cite{Mischler:20062}, and moments methods. That would perhaps allow to give up the rather unphysical   $L^p$ hypothesis on the initial condition for the proof of Haff's Law.
								
	\section*{Acknowledgement}
		I would like to thanks Francis Filbet and Cl\'ement Mouhot for fruitful discussions and comments on this article and also for their careful reading. 
		I also like to thanks Bertrand Lods for his encouragements, careful reading, and for pointing out a mistake in a proof in a preliminary version of the work.
		Finally, I want to thanks the anonymous referees for their very interesting remarks about the manuscript.

	\appendix
	
	\section{Some Results about the Granular Gases Equation}
	
		In this appendix we will present some important results concerning the inelastic Boltzmann equation, with or without drift term, that we used in this paper.
		
		Let us start by the resolution of the Cauchy problem (\ref{homo}-\ref{CI}) for an anomalous gas. 

		\begin{theorem}[\cite{Mischler:20061}, Theorem 1.4]  \label{thmCauchyPb}
			Let $f_{in}$ be a nonnegative distribution satisfying
			\begin{align*}
				f_{in} \in L^1_3  \cap \mathcal G \text{ and } & f_{in} \in BV_4 \cap L_5^1.
			\end{align*}
			Then, the following results hold for a cross section given by \ref{ba} and \ref{HS}:
			\begin{enumerate}[label=(\roman{*})]
				\item the cooling time $T_c$ is well define and positive; 
				\item for each $T \in ]0, T_c[$, there exists a unique solution $f \in \mathcal C(0,T;L_2^1) \cap L^\infty (0,T; L_3^1)$ to the initial value problem \eqref{homo}-\eqref{CI}. Such a solution is nonnegative, mass and momentum conservative, and kinetic energy dissipative;
				\item the energy of $f$ is subject to the following asymptotic behaviour
				\begin{equation} \label{asymptConv}
					\mathcal{E}(f)(t) \rightarrow 0 \text{ and } f(t,\cdot) \rightharpoonup \delta_{v=0} \text{ when } t \rightarrow T_c,
				\end{equation}
				where the convergence of $f$ occurs for weak-* topology of $\mathcal{M}^1(\mathbb{R}^d)$.
			\end{enumerate}
		\end{theorem}
		
		We then give a result stating boundedness of solutions to the granular gases equation with a drift term, together with the existence of steady-states for this equation.

		\begin{theorem}[\cite{Mischler:20062}, Theorems 1.1 and 1.3] \label{thmMiMoE}
			Let $g_{in} \in \mathcal G \cap L^p$ for a fixed $p > 1$ be an initial datum for \eqref{drift} with constant restitution coefficient $e$. If $g$  is solution to the associated Cauchy problem, then
			\begin{equation} \label{haffg}
				0 < c_0 \leq \mathcal{E}( g)(s) \leq c_1 < \infty, \, \forall s \geq 0.
			\end{equation}
			Moreover, there exists a \emph{self-similar profile} $0 \leq G \in L_2^1 \cap \mathcal G$:
			\begin{equation*}
				 \tau \, \nabla_w \cdot (w G ) - Q_e (G ,G ) = 0.
			\end{equation*}
		\end{theorem}
		
		Then, this result was extended, proving uniqueness and attractiveness of the steady-states, under a weak inelasticity assumption.
		
		\begin{theorem}[\cite{mischler:20091}, Theorem 1.1 \emph{(i)} and \emph{(iv)}]\label{thmMiMoAttract}
			There exists a constructive $e_* \in (0,1)$ such that for all $e \in [e_*, 1)$, the self-similar profile $G_e$ from Theorem \ref{thmMiMoE} is unique and globally attractive on bounded subsets of $L^1_3$: for any $M >0$, there exists $\eta \in (0,1)$ such that if
			\begin{align*}
				g_{in} \in L^1_3 \cap\mathcal G, && \| g_{in}\|_{L^1_3} \leq M_0,
			\end{align*}
			then the solution $g$ to equation \eqref{drift} satisfies
			\begin{equation} \label{attractProf}
				\left \| g(t, \cdot) - G_e \right \|_{L^1_2} \leq e^{-(1-\eta) \nu_e t},
			\end{equation}
			where $\nu_e = \tau_e+ \mathcal{O}\left (\tau_e^2\right )$.
		\end{theorem}
				
		The last important result concerns estimates of the moments of the granular collision operator $Q_e$. 
		For a nonnegative distribution $f$ and $p \geq 0$, we will define the $2p^{th}$ moment of $f$ and $Q_e(f,f)$ by
			\begin{align*}
				m_p := \int_{\mathbb{R}^d} f(v) \, |v|^{2p} \, dv, && \mathcal Q_p := \int_{\mathbb{R}^d} Q_e(f,f)(v) \, |v|^{2p} \, dv.
			\end{align*}
		This Theorem was first proved for constant cross section and restitution coefficient in \cite{bobylev:2004} and then extended to the more general case of \emph{Grad's cut-off} assumption kernels with variable restitution coefficient in~\cite{Alonso:2009vacuum}. 
		It states that:
		
		\begin{theorem}[\cite{bobylev:2004}, Lemma 3 and~\cite{Alonso:2009vacuum}, Proposition 2.6]\label{thmBGPmoments}
			For any $p \geq 1$, there exists an explicit constant $\gamma = \gamma \,(p, b) \in (0,1)$ such that
			\begin{equation*}
				\mathcal Q_p \leq -(1-\gamma) m_{p+1/2} + \gamma S_p,
			\end{equation*}
			where $S_p$ is given by
			\begin{equation*}
				S_p = \sum_{k=1}^{\left[\frac{p+1}{2}\right]} \begin{pmatrix} p \\ k \end{pmatrix}(m_{k+1/2} \, m_{p-k} + m_k \, m_{p-k+1/2}),
			\end{equation*}
			and the binomial coefficients for noninteger $p \geq 0$ have been defined by
			\begin{align*}
				\begin{pmatrix} p \\ k \end{pmatrix} = \frac{p \,(p-1)\cdots(p-k+1)}{k!}, \ k \geq 1, && \begin{pmatrix} p \\ 0 \end{pmatrix} = 1.
			\end{align*}
		\end{theorem}

	\bibliographystyle{siam}
	\bibliography{/home/Tom/Documents/These/includes/biblio}
	
\end{document}